\documentclass{amsart}   

\usepackage{amssymb}
\usepackage{amsmath}
\usepackage{amsthm}

\usepackage{url}

\usepackage{tikz, latexsym}
\usepackage{float}
\usetikzlibrary{arrows,snakes,backgrounds,positioning,shapes.geometric,shapes.misc,calc,decorations.pathmorphing}

\newtheorem{theorem}{Theorem}
\newtheorem{lemma}[theorem]{Lemma}
\newtheorem{cor}[theorem]{Corollary}
\newtheorem{prop}[theorem]{Proposition}

\newtheorem{claim}{Claim}

\newcommand{\Q}{\mathbb{Q}}

\newcommand{\A}{\mathcal{A}}
\newcommand{\K}{\mathcal{K}}
\newcommand{\B}{\mathcal{B}}

\newcommand{\F}{\mathcal{F}}
\newcommand{\M}{\mathcal{M}}

\newcommand{\SK}{\mathcal{S}}

\newcommand{\zerof}{\mathbf{0}}

\newcommand{\tr}{\ge_T}
\newcommand{\te}{=_T}

\newcommand{\ctm}{\mathfrak{c}}

\newcommand{\cl}[1]{\overline{#1}}

\newcommand{\supp}{\mathrm{supp} \, }

\newcommand{\add}[1]{\mathop{\mathrm{add}}(#1)}
\newcommand{\cof}[1]{\mathop{\mathrm{cof}}(#1)}

\newcommand{\spec}[1]{\mathop{\mathrm{spec}}(#1)}

\newcommand{\SigOm}{\Sigma (\omega^{\omega_1})}

\title{The Tukey Order and Subsets of $\omega_1$}

\author{Paul Gartside}
\address{Department
    of Mathematics, University of Pittsburgh, Pittsburgh, PA~15260, USA}
\email{gartside@math.pitt.edu} 

\author{Ana Mamatelashvili}
\address{Department of Mathematics and Statistics, Auburn University, Auburn, AL~36849, USA}
\email{azm0105@auburn.edu}



\begin{document}

\begin{abstract}
One partially ordered set, $Q$, is a \emph{Tukey quotient} of another, $P$, if there is a map  $\phi : P \to Q$ carrying cofinal sets of $P$ to cofinal sets of $Q$. Two partial orders which are mutual Tukey quotients are said to be \emph{Tukey equivalent}. Let $X$ be a space and denote by $\K(X)$ the set of compact subsets of $X$, ordered by inclusion. The principal object of this paper is to analyze the Tukey equivalence classes of $\K(S)$ corresponding to various subspaces $S$ of $\omega_1$, their Tukey invariants, and hence the Tukey relations between them. It is shown that $\omega^\omega$ is a strict Tukey quotient of $\Sigma(\omega^{\omega_1})$ and thus we distinguish between two Tukey classes  out of Isbell's ten partially ordered sets from \cite{Isb}. The relationships between Tukey equivalence classes of $\K(S)$, where $S$ is a subspace of $\omega_1$, and $\K(M)$, where $M$ is a separable metrizable space, are revealed. Applications are given to function spaces.
\end{abstract}

\maketitle

\section{Introduction}

A partially ordered set (poset) is \emph{directed} if every two elements have an upper bound. One directed partially ordered set $Q$ is a \emph{Tukey quotient} of another, $P$, denoted $P\tr Q$, if there is a map $\phi : P\to Q$, called a \emph{Tukey quotient}, that takes cofinal subsets of $P$ to cofinal subsets of $Q$. Posets that are mutual Tukey quotients are said to be \emph{Tukey equivalent}. The Tukey relation distinguishes between cofinal structures of posets and Tukey equivalence classes are sometimes called cofinality types. Introduced to study Moore-Smith convergence in topology \cite{MooreSmith, Tukey}, Tukey quotients and equivalence are fundamental notions of order theory, and are being actively investigated, especially in connection with partial orders arising naturally in analysis and topology \cite{DobrTod1, DobrTod2, Frem, Fr2, Fr3, LouVel, Mil, MooreSol, RaghTod, SolTod1, SolTod2}. 

The quest for interesting new Tukey classes began with Isbell's work in \cite{Isb}. In this paper he presented ten partially ordered sets. The first five of these posets, namely, $1$, $\omega$, $\omega_1$, $\omega \times \omega_1$ and $[\omega_1]^{<\omega}$, previously known to be Tukey-distinct, are of size $\leq \omega_1$. The other five partially ordered sets are of size $\ctm$. Isbell showed that the extra five posets provide at least two additional Tukey classes because two of them, namely $\omega^\omega$ and $\Sigma(\omega^{\omega_1})$,  have a certain Tukey invariant property which the other three do not possess (here $\Sigma(\omega^{\omega_1})$ is the subset of $\omega^{\omega_1}$ consisting of all elements which are eventually constantly equal to $0$). Up until now the Tukey classes of $\omega^\omega$ and $\Sigma(\omega^{\omega_1})$ have remained un-distinguished. 

Isbell asked a natural question: how many Tukey classes are there of size $\leq  \omega_1$? Todor\v{c}evi\'{c} in \cite{Tod1} showed that depending on set theoretic assumptions $1$, $\omega$, $\omega_1$, $\omega \times \omega_1$ and $[\omega_1]^{<\omega}$ might be the only such posets or there might be $2^{\omega_1}$, the largest possible number. 
In order to show the existence of $2^{\omega_1}$-many distinct Tukey classes of posets of size $\omega_1$, Todor\v{c}evi\'{c} considered certain special partially ordered sets. For a space $X$, let $\K(X)$ denote the collection of all compact subsets of $X$ ordered by set inclusion. Todor\v{c}evi\'{c} used Tukey classes $\K(S)$, where $S$ is a subset of $\omega_1$. 

The main goal of this paper is to understand the Tukey classes $\K(S)$, where $S$ is a subset of $\omega_1$ -- what they are, and how they are related. Note that each of  $1$, $\omega$, $\omega_1$, $\omega \times \omega_1$ and $[\omega_1]^{<\omega}$, as well as $\omega^\omega$ and $\Sigma(\omega^{\omega_1})$ have this form. To see this, let $S_0$ be the set of all isolated points in $\omega_1$, let $S_1 = (\omega \cdot \omega +1) \backslash \{ \omega \cdot n : n\in \omega \}$ and let $S_2$ be $S_0$ together with all elements of $\omega_1$ that are limits of limit ordinals. Then $\omega_1 \te \K(\omega_1)$, $\omega \times \omega_1 \te \K(\omega_1 \backslash \{\omega\})$, $[\omega_1]^{<\omega} \te \K(S_0)$, $\omega^\omega = \K(S_1)$, and $\Sigma (\omega^{\omega_1}) \te \K(S_2)$.

In Section~\ref{structure_S} we study how posets corresponding to various subsets of $\omega_1$ are divided between different Tukey equivalence classes. We then study the Tukey relations between these classes. In particular, we show that $\omega^\omega$ is a strict Tukey quotient of $\Sigma(\omega^{\omega_1})$. In order to distinguish between Tukey classes, in Section~\ref{order_properties} we  consider Tukey invariants of $\K(S)$'s such as: cofinality, additivity and calibres. In Section~\ref{background} we present the relative Tukey ordering, a more general version of the Tukey ordering, as well as related definitions and basic lemmas.  

In the last section we  give applications. Note that each bounded subset of $\omega_1$ is Polish (i.e. separable and completely metrizable). Posets of the form $\K(M)$, where $M$ is separable and metrizable have been studied extensively \cite{Chris, Fr2, KM}. In works of Cascales, Orihuela and Tkachuk a special interest has been given to investigating properties of spaces $X$ such that $\K(M) \tr \K(X)$ for some separable metrizable $M$ \cite{CO1, cot}. We completely solve the problem of  which sets $S\subseteq \omega_1$ have the property that $\K(M) \tr \K(S)$ for some separable metrizable $M$. Recently, in \cite{KM}, the authors constructed $2^\ctm$-many Tukey classes of $\K(M)$'s, with $M$ separable and metrizable, and used them to show that there are large families of function spaces corresponding to these $M$'s. Here we use Todor\v{c}evi\'{c}'s $2^{\omega_1}$-many distinct Tukey classes of $\K(S)$'s, where $S$ is a subset of $\omega_1$, to derive similar results about function spaces.

\section{Definitions and Basic Properties}\label{background}

Let $P$ and $Q$ be directed partially ordered sets. Let $P'$ be a subset of $P$ and $Q'$ be a subset of $Q$. A subset $C$ of $P$ is \emph{cofinal for $P'$} (in $P$) if whenever $p'$ is in $P'$ there is a $c$ in $C$ such that $p' \le c$ (and similarly for $Q$).
For pairs $(P',P)$ and $(Q',Q)$,  a map $\phi : P \to Q$ is a \emph{relative Tukey quotient} if $\phi$ takes subsets of $P$ cofinal for $P'$ to sets cofinal for $Q'$ in $Q$, and we write $(P',P) \tr (Q',Q)$. Pairs that are mutual relative Tukey quotients are said to be Tukey equivalent. The relative Tukey ordering was introduced in \cite{KM}. Note that $(P,P) \tr (Q,Q)$ if and only if $P \tr Q$. So no ambiguity arises if  we abbreviate $(P,P) \tr (Q',Q)$ to $P \tr (Q',Q)$, and similarly when $Q'=Q$. In this section we will give some basic connections of the (relative) Tukey ordering with order properties of posets. Sketches of proofs of some of these results were presented in \cite{KM} and we omit them here. Full proofs can be found in \cite{Anas_thesis}.  

There is a dual form of relative Tukey quotients. Call $\psi : Q' \rightarrow P'$ a relative Tukey map from $(Q', Q)$ to $(P', P)$  if and only if  for any $U\subseteq Q'$ unbounded in $Q$, $\psi(U) \subseteq P'$ is unbounded in $P$. Recall that a poset $P$ is Dedekind complete  if and only if  every  subset of $P$ with an upper bound has a least upper bound. 
\begin{lemma}\label{trtm}
(1) There is a relative Tukey quotient $\phi $ from $(P', P)$ to $(Q', Q)$ if and only if there is relative Tukey map $\psi$ from $(Q', Q)$ to $(P', P)$.

(2) If $(P',P) \tr (Q',Q)$ and  $Q$ is Dedekind complete then there is a Tukey quotient witnessing this that is order-preserving. Conversely, if $\phi : P \to Q$ is order-preserving and $\phi(P')$ is cofinal for $Q'$ in $Q$, then $\phi$ is a relative Tukey quotient from $(P',P)$ to $(Q',Q)$. 

(3) If $C$ is a cofinal set of a poset $P$ then $C$ and $P$ are Tukey equivalent.
\end{lemma}
Note that $\K(X)$ is Dedekind complete for any space $X$. Hence in this paper -- where our posets are always either a $\K(S)$, for some $S \subseteq \omega_1$, or a $\K(M)$ where $M$ is separable and metrizable -- we will routinely assume, by part (2) of the above lemma,  that any given Tukey quotient is order-preserving. Dedekind completeness can be useful in other ways.

\begin{lemma}\label{combine} For a Dedekind complete poset $P$, suppose $P = \bigcup_{\alpha \in \kappa} P_\alpha$ and for each $\alpha$ we have $Q \tr (P_\alpha ,P)$. Then $Q \times [\kappa]^{<\omega} \tr P$.
\end{lemma}
\begin{proof}
As $P$ is Dedekind complete, for each $\alpha < \kappa$, fix an order-preserving $\phi_\alpha : Q \to P$ such that $\phi_\alpha (Q)$ is cofinal for $P_\alpha$ in $P$. Define $\phi : Q \times [\kappa]^{<\omega} \to P$ by $\phi(q,F) = \sup \{  \phi_\alpha (q) : \alpha \in F\}$, which is well-defined since $P$ is directed and Dedekind complete.

Then $\phi$ is clearly order-preserving. If $p$ is any element of $P$, then $p$ is in $P_\alpha$, for some $\alpha$. Pick $q$ from $Q$ such that $\phi_\alpha (q) \ge p$. Then $\phi(q,\{\alpha\}) = \phi_\alpha (q) \ge p$, and thus $\phi$ has cofinal image.
\end{proof}

\subsection{Cofinality and additivity} Define the \emph{cofinality} of $P'$ in $P$ to be $\cof{P',P} = \min \{ |C| : C$ is cofinal for $P'$ in $P\}$.
Define the \emph{additivity} of $P'$ in $P$ to be $\add{P,P'} = \min \{ |S| : S \subseteq P'$ and $S$ has no upper bound in $P\}$.
Then $\cof{P}=\cof{P,P}$ and $\add{P}=\add{P,P}$ coincide with the usual notions of cofinality and additivity of a poset.
\begin{lemma}\label{cof_add}
If $(P',P) \tr (Q',Q)$ then $\cof{P',P} \ge \cof{Q',Q}$, and $\add{P',P} \le \add{Q',Q}$.
\end{lemma}
\begin{lemma}\label{cof_dom}
The cofinality of $P'$ in $P$ is $\leq \kappa$ if and only if $[\kappa]^{<\omega} \tr (P',P)$.
\end{lemma}
%



\subsection{Calibres and Spectra.} Let $\kappa \ge \mu \ge \lambda$ be cardinals. We say that $P'$ has \emph{calibre $(\kappa,\lambda,\mu)$} in $P$ (or, $(P', P)$ has caliber $(\kappa,\lambda,\mu)$) if for every $\kappa$-sized subset $T$ of $P'$ there is a $\lambda$-sized subset $T_0$ such that every $\mu$-sized subset $T_1$ of $T_0$ has an upper bound in $P$. When $P'=P$ this coincides with the standard definition of calibre of a poset. `Calibre $(\kappa, \lambda, \lambda)$' is abbreviated to `calibre $(\kappa, \lambda)$' and `calibre $(\kappa, \kappa)$' is abbreviated to `calibre $\kappa$'. 
\begin{lemma}\label{calibregoesdown}
If $(P',P) \tr (Q',Q)$, $P'$ has calibre $(\kappa,\lambda,\mu)$ in $P$ and $\kappa$ is regular, then $Q'$ has calibre $(\kappa,\lambda,\mu)$ in $Q$.
\end{lemma}

\begin{lemma}\label{caltuk}
Suppose $\kappa$ is regular. Then (1) $P'$ fails to have calibre $\kappa$ in $P$ if and only if $(P',P)\tr \kappa$, (2) If $(P',P)\tr [\kappa]^{< \lambda}$ then $P'$ fails to have calibre $(\kappa, \lambda)$ and the converse is true if $\add{P'} \geq \lambda$ (equivalently, all subsets of $P'$ of size $<\lambda$ are bounded in $P'$). 
\end{lemma}
\begin{proof}
Clearly, $\kappa$ does not have calibre $\kappa$. So, by Lemma~\ref{calibregoesdown}, $(P',P)\tr \kappa$ implies that $P'$ does not have calibre $\kappa$ in $P$. 

Similarly, $[\kappa]^{< \lambda}$ does not have calibre $(\kappa, \lambda)$ ($\{ \{\alpha\} : \alpha < \kappa \}$ is a $\kappa$-sized collection in $[\kappa]^{< \lambda}$ but none of its $\lambda$-sized subcollections is bounded in $[\kappa]^{< \lambda}$) and therefore $(P',P)\tr [\kappa]^{< \lambda}$ implies that $P'$ fails to have calibre $(\kappa, \lambda)$ in $P$. 

On the other hand, suppose $P'$ fails to have calibre $\kappa$ in $P$. Then there exists $\kappa$-sized $P_0 \subseteq P'$ such that all $\kappa$-sized subsets of $P_0$ are unbounded. Let $\psi : \kappa \to P_0\subseteq P$ be a bijection. Since $\kappa$ is regular, all unbounded subsets of $\kappa$ are $\kappa$-sized and their images are unbounded as well.  Therefore $\phi$ is a relative Tukey map. 

Similarly, suppose $P'$ fails to have calibre $(\kappa, \lambda)$ in $P$. Then there exists $\kappa$-sized $P_1 \subseteq P'$ such that all $\lambda$-sized subsets of $P_1$ are unbounded in $P$. Let $j : \kappa \to P_1\subseteq P$ be a bijection. Since $\add{P'} \geq \lambda$ we can define $\psi :  [\kappa]^{< \lambda} \to P'$ by $\psi(F) = $ an upper bound of $\{ j(\alpha): \alpha \in F  \}$ in $P'$.  Suppose $U$ is an unbounded subset of $[\kappa]^{< \lambda}$. This means that $\bigcup U$ has size $\geq \lambda$ and therefore $\{ j(\alpha) : \alpha \in \bigcup U  \}$, a subset of $P_1$ of size $\geq \lambda$,  is also unbounded in $P$. Since any bound of $\{ \psi(F) : F\in U\}$ is also a bound of $\{ j(\alpha) : \alpha \in \bigcup U  \}$ in $P$, we get that $\{\psi(F) : F\in U \}$ is unbounded and $\psi$ is a relative Tukey map. 
\end{proof}

The next lemma gives relative versions of known facts on productivity of calibres. The straightforward proof is left to the reader. 

\begin{lemma}\label{cal_product2}
(1) Both $(P',P)$ and $(Q',Q)$ have calibre $(\kappa, \kappa, \mu)$  if and only if $(P'\times Q',P\times Q)$ has calibre $(\kappa, \kappa, \mu)$. 

(2) If $P'$ (or $Q'$) fails to have calibre $(\kappa, \lambda, \mu)$ in $P$ (respectively, in $Q$) then $P'\times Q'$ also fails to have calibre $(\kappa, \lambda, \mu)$ in $P \times Q$. 
\end{lemma}

It is clearly of interest to consider the collection of all regular cardinals $\kappa$ such that a  poset $P$ has calibre $\kappa$, but since any regular cardinal $\kappa > |P|$ is a calibre of $\kappa$ this collection is unbounded. Consequently, we instead   study the \emph{spectrum} of $P$, $\spec{P}$,  the set of all regular $\kappa$ which are \emph{not} calibres of $\kappa$. Equivalently, $\spec{P}=\{ \kappa : \kappa \ \text{is a regular cardinal and} \ P\tr \kappa  \}$. Since calibre $\kappa$ is productive we immediately get the following corollary. 
\begin{cor}\label{prod_spec}
$\spec{P_1 \times P_2} = \spec{P_1} \cup \spec{P_2}$.
\end{cor}
And transitivity of $\leq_T$ yields:
\begin{lemma}\label{specNN_in_all}
If $Q\leq_T P$ then $\spec{Q} \subseteq \spec{P}$.
\end{lemma}

Lastly we present a lemma that establishes a close relationship between the spectrum of a poset and its additivity and cofinality. (We write $[\lambda,\mu]^r$ for $\{\kappa : \kappa \text{ is regular and } \lambda \le \kappa \le \mu\}$.) 

\begin{lemma}\label{add_cof_spec}
Let $P$ be a directed poset without the largest element. Then \\
$\add{P}$ and $\cof{\cof{P}}$ are elements of $\spec{P}$ and $\spec{P} \subseteq [\add{P}, \cof{P}]^r$.
\end{lemma}



\subsection{The Pair $(X, \K(X))$} 
If we identify elements of $X$ with corresponding singletons we can view $X$ as a subset of $\K(X)$. Most relative Tukey pairs  considered in this paper are  of the form $(X, \K(X))$. Note that a subset of $\K(X)$ is cofinal for $X$ in $\K(X)$ if it covers $X$. So $(X, \K(X)) \tr (Y, \K(Y))$ means that there is an order-preserving map from $\K(X)$ to $\K(Y)$ that takes (compact) covers of $X$ to (compact) covers of $Y$.  While $\K(X) \tr (Y, \K(Y))$ means that there is an order-preserving map from $\K(X)$ to $\K(Y)$ that takes cofinal subsets of $\K(X)$ to covers of $Y$. The next lemma is particularly useful.
\begin{lemma}\label{closed}
Let $A$ be a closed subspace of a space $X$. Then $(X, \K(X)) \tr (A,\K(A))$, $\K(X) \tr \K(A)$, and $\K(X) \tr (A,\K(A))$.
\end{lemma}
For spaces $X$ and $Y$, let $X\oplus Y$ denote the free union of $X$ and $Y$. Then clearly, $\K(X\oplus Y) = \K(X) \times \K(Y)$.

\subsection{Notation} 
For the rest of this paper $S$ will always denote a subset of $\omega_1$ and $M$ will denote a separable metrizable space. Let $\SK$ be the set of all homeomorphism classes of subsets of $\omega_1$. Let $\M$ be the set of all homeomorphism classes of separable metrizable spaces. Let $\K(\SK) = \{ [\K(S)]_T: S\in \SK \}$ and $\K(\M) = \{ [\K(M)]_T: M\in \M \}$.

\subsection{Separable Metrizable Spaces}  As mentioned above, each bounded subset of $\omega_1$ is Polish (because it is countable, scattered and metrizable). Therefore we briefly describe the Tukey classes of Polish spaces.  In \cite{Chris} Christensen proved that, for a separable metrizable space $M$, $\omega^\omega \tr \K(M)$ if and only if $M$ is Polish. Based on this result one can obtain the following description of the initial segment of Tukey classes in $\K(\M)$.  

\begin{theorem}[Christensen, Fremlin \cite{Fr2}]\label{M_initial}
Below $M$ is a separable metrizable space.

(1) The minimum Tukey equivalence class in $\K(\M)$ is $[1]_T$, and $\K(M)$ is in this class if and only if $M$ is compact;  

(2) it has a unique successor, $[\omega]_T$, which consists of all $\K(M)$ where $M$ is locally compact but not compact; and 

(3) this has $[\omega^\omega]_T=\{\K(M): M$ is Polish, not locally compact$\}$ as a unique successor. 
\end{theorem}

The simplest example of a Polish, non-locally compact space is the metric fan. The following observation immediately follows from the above theorem but it can also be shown independently of it using the fact that the metric fan embeds as a closed subspace in $M$ if and only if $M$ fails to be locally compact.  

\begin{cor}\label{not_loc_cpt} Let $M$ be separable, metrizable and non-locally compact. Then $\omega^\omega \leq \K(M)$.  
\end{cor}

A similar statement is true for arbitrary subsets of $\omega_1$. Notice that since $\omega_1$ is locally compact, a subset $S$ of $\omega_1$ is locally compact if and only if $S$ is open in its closure, which happens if and only if $\cl{S} \backslash S$ is closed in $\omega_1$. It turns out that when we study various groups of subsets of $\omega_1$ (stationary, co-stationary, closed and unbounded, etc.) it is convenient to work with $\cl{S} \backslash S$. 
\begin{lemma}\label{closure_and_fan}
Let $S$ be a subset of $\omega_1$. Then the following are equivalent:

(1) $S$ is locally compact (i.e. $\cl{S}\backslash S$ is closed),

(2) $S$ does not contain a metric fan as a closed subspace, and

(3) $\omega^\omega \nleq_T \K(S)$.
\end{lemma}
\begin{proof}
Clearly, $S$ is locally compact if and only if $S\cap [0,\alpha]$ is locally compact for each  $\alpha \in \omega_1$. Also, since a metric fan is countable and each $S\cap [0,\alpha]$ is closed, $S$ contains a metric fan as a closed subspace if and only if there is some $\alpha \in \omega_1$ such that $S\cap [0,\alpha]$ contains a metric fan as a closed subspace. For any $\alpha \in \omega_1$, $S\cap [0,\alpha]$ is separable and metrizable and therefore it is locally compact if and only if it contains a metric fan as a closed subspace. Therefore (1) and (2) are equivalent. 

Lastly, if $S$ contains a metric fan as a closed subspace then $\omega^\omega \leq \K(S)$. If not, then $\cl{S} \backslash S$ is closed, which means $\K(S) $ is Tukey equivalent to either $1$, $\omega$ or $\omega \times \omega_1$, none of which are above $\omega^\omega$ in the Tukey order.
\end{proof}





\section{Order properties of elements of $\K(\SK)$}\label{order_properties}

 In this section we will consider the Tukey invariants -- cofinality, additivity, calibres and spectrum -- introduced  above,  of elements of $\K(\SK)$. 

\subsection{Size and bounds} Every compact subset $K$ of $\omega_1$ is contained in some initial segment, $[0,\alpha]$. Consequently, $\K(\omega_1) = \bigcup \{ \K([0,\alpha]) : \alpha < \omega_1\}$, and for any $S \subseteq \omega_1$, $\K(S) = \bigcup \{ \K([0,\alpha] \cap S) : \alpha < \omega_1\}$. So  $|\K(S)| \leq \ctm \times \omega = \ctm$. Whenever $S$ contains a convergent sequence together with its limit, $|K(S)| \geq \ctm$ and since compact subsets of discrete spaces are finite we immediately have the following lemma. 

\begin{lemma} Let $S$ be a subset of $\omega_1$. The possible  sizes of $\K(S)$ are as follows.
 If $S$ is discrete then:
 (a) $\K(S)$ is finite if $S$ is finite,
 (b) $|\K(S)| = \omega$ if $S$ is countable infinite, and 
 (c) $|\K(S)| = \omega_1$ if $S$ is uncountable.  If $S$ is not discrete then $|\K(S)| = \ctm$.
\end{lemma}

By Lemma~\ref{cof_dom} and the previous lemma we clearly have that $[\ctm]^{<\omega}$ bounds each $\K(S)$ from above. We  refine this upper bound for $\K(S)$. Recall $\K(S)=\bigcup_{\alpha \in \omega_1} \K(S \cap [0,\alpha])$. Since each $S \cap [0,\alpha]$ is Polish we have $\K(S \cap [0,\alpha]) \leq_T \omega^\omega$, and so we know by Lemma~\ref{combine} that $\K(S) \leq_T \omega^\omega \times [\omega_1]^{<\omega}$. By Lemma~\ref{not_loc_cpt}, if $S$ is not locally compact then $\omega^\omega$ is a lower bound of $\K(S)$. Now suppose $S$ is unbounded. Enumerate $S$ as $\{\beta_\alpha : \alpha\in \omega_1\}$. Since all compact subsets of $\omega_1$ are countable, the map $\psi : \omega_1 \to \K(S)$ given by $\psi(\alpha) = \{\beta_\alpha\}$ is a Tukey map. Hence in this case $\omega_1$ is a lower bound of $\K(S)$.
\begin{lemma}\label{bounds}
Let $S$ be a subset of $\omega_1$. Then
(a) $\K(S) \leq \omega^\omega \times [\omega_1]^{<\omega}$, 
(b) $\omega^\omega \leq_T \K(S)$ whenever $S$ is non-locally compact, and
(c) $\omega_1 \leq_T \K(S)$ whenever $S$ is unbounded.
\end{lemma}




\subsection{Additivity and cofinality}

If $X$ is compact then $\add{\K(X)}$ is undefined. We compute $\add{\K(S)}$ otherwise.
\begin{lemma} Let $S$ be a subset of $\omega_1$. 
If $S$ is closed and unbounded then $\add{\K(S)}=\omega_1$. If $S$ is not closed then $\add{\K(S)}=\omega$.
\end{lemma}
\begin{proof}
If $S$ is closed and unbounded,  every countable subset of $\K(S)$ is bounded, but collection of all singletons of $S$ is not. So,  $\add{\K(S)}=\omega_1$. On the other hand if $S$ is not closed, pick a sequence $\{x_n : n \in \omega \}$ in $S$ that does not converge in $S$. Then $\{ \{x_n\}: n\in \omega  \}$ is unbounded in $\K(S)$ and $\add{\K(S)}=\omega$.
\end{proof}

We know that $\cof{\K(\omega^\omega)} = \mathfrak{d}$. In light of Lemma~\ref{bounds} we have the following. 

\begin{lemma}\label{cof} Let $S$  be  a subset of $\omega_1$. 
There are four possibilities for $\cof{\K(S)}$.

(1) If $S$ is compact then $\cof{\K(S)}=1$.

(2) If $\cl{S} \backslash S$ is closed and $S$ is bounded, then $\cof{\K(S)}=\omega$.

(3) If $\cl{S} \backslash S$ is closed and $S$ is unbounded, then $\cof{\K(S)}=\omega_1$.

(4) If $\cl{S} \backslash S$ is not closed, then $\cof{\K(S)}=\mathfrak{d}$.
\end{lemma}
\begin{proof}
Claim (1) is clear, and (2) follows from Lemma~\ref{closure_and_fan}. For (3), note that $\K(S)=\bigcup_{\alpha \in \omega_1} \K(S \cap [0,\alpha])$. But by Lemma~\ref{closure_and_fan}, $\cof{\K(S \cap [0,\alpha])}\leq \omega$ for each $\alpha$ and therefore $\cof{\K(S)}\leq \omega_1$. Since $S$ is uncountable and all compact subsets are countable $\cof{\K(S)}\geq \omega_1$ and we are done. 

For (4), by Lemma~\ref{bounds} we have $\omega^\omega \leq_T \K(S) \leq_T \omega^\omega \times [\omega_1]^{<\omega}$. Since $\cof{\omega^\omega} = \cof{\omega^\omega \times [\omega_1]^{<\omega}} = \mathfrak{d}$, we have $\cof{\K(S)}=\mathfrak{d}$.
\end{proof} 

\subsection{Spectrum of $\K(S)$}\label{spec_S}

Spectra associated with separable metrizable spaces were studied in \cite{Kcal}. Particular attention was paid to the spectrum of $\omega^\omega$, which is contained in spectra of all $\K(M)$'s and $\K(S)$'s if $M$'s and $S$'s are non-locally compact. It turns out that $\spec{\omega^\omega}$ can be consistently equal to any finite set of regular cardinals as well as some infinite sets of regular cardinals. Here we compute $\spec{\K(S)}$ in terms of $\spec{\omega^\omega}$. If $S$ is a bounded subset of $\omega_1$ then $\K(S)$ is Tukey equivalent to either $\K(1)$ or $\K(\omega)$ or $\K(\omega^\omega)$. So the interesting case for the spectrum of $\K(S)$ is when $S$ is unbounded. 

\begin{theorem}
Suppose $S\subseteq \omega_1$ is unbounded. If $\cl{S} \backslash S$ is closed then $\spec{\K(S)} = \{\omega_1\}$ and if $\cl{S} \backslash S$ is not closed then $\spec{\K(S)} = \{\omega_1\} \cup \spec{\omega^\omega}$. 
\end{theorem}
\begin{proof}
For unbounded $S$, $\omega_1 \leq_T \K(S)$. If $\cl{S} \backslash S$ is closed then $\cof{\K(S)} = \omega_1$ and $\kappa \leq_T \K(S)$ implies $\cof{\kappa}\leq \omega_1$. So $\kappa$ must be equal to $\omega_1$ and $\spec{\K(S)} = \{ \omega_1 \}$.

If $\cl{S} \backslash S$ is not closed then $\omega^\omega \times \omega_1 \leq_T \K(S) \leq_T \omega^\omega \times [\omega_1]^{<\omega}$. And $\spec{\omega^\omega \times \omega_1} \subseteq \spec{\K(S)} \subseteq \spec{\omega^\omega \times [\omega_1]^{<\omega}}$.
But from Corollary~\ref{prod_spec}, it is clear that $\spec{\omega^\omega \times \omega_1} = \{\omega_1\} \cup \spec{\omega^\omega} = \spec{\omega^\omega \times [\omega_1]^{<\omega}}$.
\end{proof}

\subsection{Calibres of $\K(S)$}\label{cal_S}
Let $S$ be a subset of $\omega_1$. Since $\K(S)$ has size at most $\ctm$ we focus on calibres $\omega_1$, $(\omega_1, \omega_1, \omega)$ and $(\omega_1, \omega)$. Calibres for $\K(S)$'s when $S$ is bounded are known using facts about calibres of $\K(M)$'s given in \cite{Kcal}. If $\K(S)$ is Tukey equivalent to $1$ or $\omega$ then $\K(S)$ has calibres $\omega_1$, $(\omega_1, \omega_1, \omega)$ and $(\omega_1, \omega)$; If $\K(S)$ is Tukey equivalent to $\omega^\omega$ then $\K(S)$ has calibre $(\omega_1,\omega)$ but has the other two calibres if and only if $\omega_1 <\mathfrak{b}$. 

Now let $S$ be unbounded. We showed that $\omega_1 \leq_T \K(S)$ and therefore $\K(S)$ fails to have calibre $\omega_1$. The case of calibre $(\omega_1, \omega)$ has already been settled by Todor\v{c}evi\'{c} in \cite{Tod1}. Recall that a subset of an ordinal is called stationary if and only if it meets every closed and unbounded subset of $\omega_1$. Using the fact that $\K(S)\geq [\omega_1]^{<\omega}$ if and only if $\K(S)$ does not have calibre $(\omega_1, \omega)$, Todor\v{c}evi\'{c}'s theorem becomes: 
 
\begin{lemma}[Todor\v{c}evi\'{c}]\label{cal1,0} Let $S\subseteq \omega_1$ be unbounded. Then $\K(S)$ has calibre $(\omega_1, \omega)$ if and only if $S$ is stationary if and only if $\K(S) \ngeq [\omega_1]^{<\omega}$.
\end{lemma}

In fact, Todor\v{c}evi\'{c} shows that if $S$ is not stationary then $S$ contains an uncountable closed discrete subset, which gives an uncountable collection of singletons such that any infinite subcollection is unbounded in $\K(S)$. 

\begin{lemma}[Todor\v{c}evi\'{c}]\label{stat_discrete}
Suppose $S$ is unbounded, $C$ is closed, unbounded and $S\cap C = \emptyset$. Then there exist strictly increasing sequences $\{s_\alpha : \alpha <\omega_1\} \subseteq S$ and $\{ c_\alpha : \alpha < \omega_1 \} \subseteq C$ such that for each $\alpha <\omega_1$, $s_\alpha < c_\alpha < s_{\alpha +1}$. Hence, $S$ contains an uncountable closed discrete subset.
\end{lemma} 


Next we show exactly when $\K(S)$ has calibre $(\omega_1, \omega_1, \omega)$. Recall that a subset of $ \omega_1$ is called co-stationary if its complement is stationary. Equivalently $S$ is not co-stationary if and only if it contains a cub (closed and unbounded) set. Note that if $S$ is unbounded and $\cl{S}\backslash S$ is bounded then $S$ is co-stationary. In particular, $S \backslash (\sup (\cl{S}\backslash S)+1)$ is a cub subset of $S$. 

\begin{lemma}\label{calibre-club}
Let $S\subseteq \omega_1$ be unbounded. Then  $\K(S)$ has calibre $(\omega_1,\omega_1,\omega)$ if and only if $\cl{S}\backslash S$ is bounded and either $\cl{S}\backslash S$ is closed or $\omega_1<\mathfrak{b}$.
\end{lemma}

\begin{proof} Let $S$ be an unbounded subset of $\omega_1$ and suppose $\cl{S}\backslash S$ is bounded. Let $\alpha = (\sup (\cl{S}\backslash S)+1)$. Then $S \backslash \alpha$ is closed and unbounded in $\omega_1$ and $S$ is the free union of $S\cap \alpha$ and $S\backslash \alpha$, which implies $\K(S)=\K(S\cap \alpha) \times \K(S\backslash \alpha)=\K(S\cap \alpha) \times \omega_1$. Clearly, $\omega_1$ has calibre $(\omega_1, \omega_1, \omega)$. So by Lemma~\ref{cal_product2}, $\K(S)$ has calibre $(\omega_1,\omega_1,\omega)$ if and only if $\K(S\cap \alpha)$ does. Since $S\cap \alpha$ is Polish, $\K(S\cap \alpha)$ has calibre $(\omega_1, \omega_1, \omega)$ if and only if either $\cl{S\cap \alpha}\backslash (S\cap \alpha)$ is closed or $\omega_1 <\mathfrak{b}$. Then, by the fact that $\cl{S}\backslash S=\cl{S\cap \alpha}\backslash (S\cap \alpha)$, $\K(S)$ has calibre $(\omega_1,\omega_1,\omega)$ if and only if either $\cl{S}\backslash S$ is a closed or $\omega_1<\mathfrak{b}$. 

What is left to show is that if $\K(S)$ has calibre $(\omega_1,\omega_1,\omega)$ then $\cl{S}\backslash S$ is bounded. Suppose $S$ has calibre $(\omega_1,\omega_1,\omega)$. First we show that $S$ contains a cub set. Let $\K = \{ \{\alpha\} : \alpha \in S \}$. Then there is an uncountable $\K'\subseteq \K$ with every countable subset bounded in $\K(S)$. If we let $T= \bigcup \K'$, then every limit point of $T$ lies in $S$: otherwise pick $\beta \in \omega_1\backslash S$ and  $\{\alpha_n\}_{n\in \omega}\subseteq T$ with $\beta=\sup\{\alpha_n\}_{n\in \omega}$. Then $\{\{\alpha_n\} : n\in \omega\}$ does not have an upper bound in $\K(S)$. So, $C=\cl{T}$ is closed and unbounded subset of $S$. 

Therefore $\omega_1 \backslash S$ is non-stationary. If $\omega_1 \backslash S$ is also unbounded, apply Lemma~\ref{stat_discrete} to $\omega_1 \backslash S$ and $C$ to get strictly increasing sequences $\{s_\alpha : \alpha <\omega_1\} \subseteq \omega_1 \backslash S$ and $\{ c_\alpha : \alpha < \omega_1 \} \subseteq C$ such that for each $\alpha <\omega_1$, $s_\alpha < c_\alpha < s_{\alpha +1}$. 

Since $\cl{S} \backslash S$ is non-stationary in $\cl{S}$, which is homeomorphic to $\omega_1$, we may assume that $\cl{S}=\omega_1$. \emph{Then all successor ordinals are in $S$}. 

Let $\alpha \in \omega_1$. Considering how $s_\alpha$'a and $c_\alpha$'s were chosen, $s_\alpha > \sup (\{s_\beta : \beta < \alpha\} \cup \{c_\beta : \beta < \alpha\})$. Since $s_\alpha$ is a limit ordinal, we can pick an increasing sequence, $\{s_{\alpha,m}\}_{m\in \omega}$, of successor ordinals in the interval $( \sup (\{s_\beta : \beta < \alpha\} \cup \{c_\beta : \beta < \alpha\}), s_\alpha)$ that converges to $s_\alpha$. 

For each infinite $\alpha \in \omega_1$ let $f_\alpha :  \alpha \to \omega$ be a bijection. Fix infinite $\alpha \in \omega_1$ and define $K_\alpha=\cl{\{s_{\sigma, f_\alpha(\sigma)} : \sigma  \in \alpha \}}$. For each $\sigma , \sigma' \in \alpha$ with $\sigma < \sigma'$, $s_{\sigma,f_\alpha(\sigma)} < c_{\sigma} < s_{\sigma',f_\alpha(\sigma')}$. So every limit point of $\{s_{\sigma, f_\alpha(\sigma)} : \sigma  \in \alpha \}$ is also a limit point of $\{c_\sigma : \sigma \in \alpha\}$ and therefore lies in $C\subseteq S$. Therefore, since $\{s_{\sigma, f_\alpha(\sigma)} : \sigma  \in \alpha \} \subseteq S$, $K_\alpha$ is in $\K(S)$.

If $T\subseteq \omega_1$ is uncountable, we will show that $\{K_\alpha :  \alpha \in T\}$ contains a countable subset that is unbounded in $\K(S)$. For this it will suffice to find $\sigma \in \omega_1$ such that $A_\sigma=\{ f_\alpha(\sigma): \alpha \in T, \ \alpha>\sigma  \}$ is infinite; because then for each $n\in A_\sigma$, we can pick $\alpha_n \in T$ with $f_{\alpha_n}(\sigma)=n$, which will imply that $\bigcup_{n\in A_\sigma} K_{\alpha_n}$ contains an infinite subset of $\{ s_{\sigma,m} : m\in \omega  \}$ and therefore is unbounded in $\K(S)$.

Suppose, to get a contradiction, that for each $\sigma \in \omega_1$ there is $n_\sigma \in \omega$ that bounds $\{ f_\alpha(\sigma): \alpha \in T, \ \alpha>\sigma  \}$. Then there is uncountable $T_1\subseteq \omega_1$ and $n\in \omega$ such that $n_\sigma = n$ for all $\sigma \in T_1$. Since $T$ and $T_1$ are uncountable, there is $\alpha \in T$ such that $\alpha \cap T_1$ is infinite. Then we have $f_\alpha (\sigma) \leq n$ for all $\sigma \in \alpha \cap T_1$, which contradicts the fact that $f_\alpha$ is a bijection.

\end{proof}

\section{Structure of $\K(\protect\SK)$}\label{structure_S}

In this section we compute the Tukey classes, $[\K(S)]_T$, that correspond to subsets $S$ of $\omega_1$.  It turns out that the class depends on whether or not a subset is bounded, locally compact, stationary and co-stationary. We then go on to determine the relations between these Tukey classes under the Tukey order.

By Theorem~\ref{M_initial}, for bounded $S\subseteq \omega_1$ the poset $\K(S)$ is Tukey equivalent to one of $1$, $\omega$ or $\omega^\omega$ if $S$ is, respectively, compact, locally compact non-compact or non-locally compact. Now we assume that $S\subseteq \omega_1$ is unbounded and consider different cases depending on the behavior of $\overline{S} \backslash S$. If $\overline{S} \backslash S$ is bounded then there is a closed unbounded $C\subseteq \omega_1$ and bounded $N \subseteq \omega_1$ such that $S = N\oplus C$ and therefore $\K(S) = \K(N) \times \K(C)$. Since any closed unbounded subset of $\omega_1$ is homeomorphic to $\omega_1$ and initial segments of $\omega_1$ are cofinal in $\K(\omega_1)$, we have $\K(S) = \K(N) \times \K(\omega_1) = \K(N) \times \omega_1$. Therefore, if $\overline{S} \backslash S$ is bounded, $\K(S)$ is Tukey equivalent to one of $1 \times \omega_1 \te \omega_1$, $\omega \times \omega_1$ or $\omega^\omega \times \omega_1$ depending on whether $S$ is closed unbounded, locally compact but not closed or non-locally compact. 

The next case to consider is when $\overline{S} \backslash S$ contains a closed unbounded set. Then $S$ is non-stationary and, by Lemma~\ref{cal1,0}, $\K(S) \tr [\omega_1]^{<\omega}$. If $S$ is non-locally compact then we also have that $\K(S)\tr \omega^\omega$ and since $[\omega_1]^{<\omega} \times \omega^\omega$ is Tukey-above all elements of $\K(\SK)$, we have $\K(S) \te [\omega_1]^{<\omega} \times \omega^\omega$. Therefore, since there exist non-locally compact non-stationary subsets of $\omega_1$ (for instance, the set of all isolated points together with $\omega \cdot \omega$), we see that $[\omega_1]^{<\omega} \times \omega^\omega$ is the largest element in $\K(\SK)$. On the other hand, if $S$ is locally compact, then $S \cap [0, \alpha] \leq_T \omega$ for all $\alpha \in \omega_1$ and, by Lemma~\ref{combine}, $\K(S) \leq_T \omega \times [\omega_1]^{<\omega} \te [\omega_1]^{<\omega}$. Therefore $\K(S) \te [\omega_1]^{<\omega}$.




The most interesting case is when $\overline{S} \backslash S$ is unbounded and does not contain a closed unbounded set. First suppose $\overline{S} \backslash S$ is non-stationary. Then $S$ contains a closed unbounded set and we have the following result.

\begin{prop}
Let $S$ be a subset of $\omega_1$ such that $\cl{S} \backslash S$ is unbounded and $S$ contains a cub set. Then $\K(S) =_T \Sigma (\omega^{\omega_1})$.
\end{prop}

\begin{proof}
Fix $S$ as above and let $C\subseteq S$ be a cub set. We want to construct a cub set $C' = \{ \beta_\alpha : \alpha \in \omega_1 \} \subseteq C$ such that for each $\alpha \in \omega_1$, $\K([\beta_\alpha, \beta_{\alpha +1}]\cap S) =_T \omega^\omega$. Suppose we have constructed the desired $\beta_\gamma$ for each $\gamma < \alpha$. First let $\alpha$ be a successor. Then since $(\cl{S} \backslash S) \cap (\alpha -1 , \omega_1)$ is not closed, there exists $\beta_\alpha \in C$ such that $[\beta_{\alpha-1}, \beta_\alpha] \cap S$ contains a metric fan as a closed subspace and therefore $\K([\beta_{\alpha-1}, \beta_\alpha] \cap S) =_T \omega^\omega$. If $\alpha$ is a limit, let $\beta_\alpha = \sup \{ \beta_\gamma: \gamma < \alpha \}$. This sequence clearly works. 

For each $K \in \K(S)$, there exists the smallest $\alpha_K \in \omega_1$ such that $K= \bigcup_{\gamma < \alpha_K} K\cap [\beta_\gamma, \beta_{\gamma+1}]$. Clearly, each $K\cap [\beta_\gamma, \beta_{\gamma+1}] \in \K([\beta_\gamma, \beta_{\gamma+1}] \cap S)$. And for any choice of $\alpha \in \omega_1$ and $K_\gamma \in \K([\beta_\gamma, \beta_{\gamma+1}] \cap S)$ for $\gamma < \alpha$, $\cl{\bigcup_{\gamma<\alpha} K_\gamma} \in \K(S)$.

We now show that $\K(S) =_T \Sigma ((\omega^\omega)^{\omega_1})$. Since $\Sigma ((\omega^\omega)^{\omega_1})$ is clearly order-isomorphic to $\Sigma (\omega^{\omega_1})$, that will complete the proof.

To show $\K(S) \tr \Sigma ((\omega^\omega)^{\omega_1})$, fix Tukey quotients $\phi_\alpha : \K([\beta_\alpha, \beta_{\alpha +1}]\cap S)  \to \omega^\omega$ for each $\alpha \in \omega_1$. Define $\phi : \K(S) \to \Sigma ((\omega^\omega)^{\omega_1})$ by $\phi (K) = \Pi_{\gamma < \alpha_K} \phi_\gamma(K\cap [\beta_\gamma, \beta_{\gamma+1}])$. Clearly, $\phi$ is order-preserving. It is also cofinal since for any choice of functions $f_\gamma \in \omega^\omega$ for $\gamma < \alpha$, there is $K_\gamma \in \K([\beta_\gamma, \beta_{\gamma+1}] \cap S)$ such that $\phi_\gamma (K_\gamma) \geq f_\gamma$. Then $\phi (\cl{\bigcup_{\gamma<\alpha} K_\gamma}) \geq \Pi_{\gamma < \alpha} f_\gamma$.

For the other direction, fix Tukey quotients $\phi'_\alpha : \omega^\omega \to \K([\beta_\alpha, \beta_{\alpha +1}]\cap S)$ for each $\alpha \in \omega_1$. Define $\phi' : \Sigma ((\omega^\omega)^{\omega_1}) \to \K(S)$ by $\phi' (\Pi_{\gamma < \alpha_K} f_\gamma)= \cl{\bigcup_{\gamma<\alpha} \phi'(f_\gamma)}$. Clearly, $\phi'$ is order-preserving and cofinal. 
\end{proof}

The remaining case  is when $S$ is stationary and co-stationary (and therefore $\overline{S} \backslash S$ is unbounded and does not contain a closed unbounded set). The following result by Todor\v{c}evi\'{c} proven in \cite{Tod1} shows that there are $2^{\omega_1}$-many classes corresponding to stationary and co-stationary subsets of $\omega_1$. 

\begin{theorem}[Todor\v{c}evi\'{c}]\label{S>S'}
Let $S$ and $S'$ be unbounded subsets of $\omega_1$. Then $\K(S)\geq \K(S')$ implies that $S\backslash S'$ is non-stationary. 
\end{theorem}

In the proof the author shows that if $S\backslash S'$ is stationary then for any function $g : \K(S') \to \K(S)$ there is a collection of \emph{singletons} in $\K(S')$ such that their image under $g$ is bounded. So, in fact, the author proves that if $S\backslash S'$ is stationary, then there is no \emph{relative} Tukey map from $S'$ to $\K(S)$. Now the fact that $\omega_1$ splits into $\omega_1$-many pairwise disjoint stationary subsets gives the following theorem. 

\begin{theorem}[Todor\v{c}evi\'{c}]\label{antichain}
There is a $2^{\omega_1}$-sized family, $\A$, of subsets of $\omega_1$ such that for distinct $S$ and $T$ from $\A$ we have $\K(S) \not\tr (T,\K(T))$ and $\K(T) \not\tr (S,\K(S))$.
\end{theorem}

Since $\K(\SK)$ has size $\leq 2^{\omega_1}$ we immediately have the following corollary.

\begin{cor}
We have $|\K(\SK)|=2^{\omega_1}$ and $\K(\SK)$ contains an antichain of size $2^{\omega_1}$.
\end{cor}

Theorem~\ref{S>S'} shows that if subsets of $\omega_1$ differ by stationary-many points, then the corresponding Tukey classes also differ. It turns out that, except in trivial cases, the converse is also true. 

\begin{prop}
Let $S$ and $T$ be a subsets of $\omega_1$ such that $\cl{S} \backslash S$ is unbounded, $S$ is stationary and $S \Delta T$ is non-stationary. Then $\K(S) \tr \K(T)$.   
\end{prop}
\begin{proof}
Let $S$ and $T$ be as above. Since $S \Delta T$ is non-stationary, it contains uncountable closed discrete subset, $D$. Enumerate $D= \{ \beta_\alpha : \alpha \in \omega_1 \} $ in increasing order. Let $\beta_{<\alpha} = \sup \{  \beta_\gamma : \gamma < \alpha\}$ for each $\alpha$ (let $\beta_{<0} = 0$). Since $\cl{S} \backslash S$ is unbounded, we can refine $D$ so that $\K([\beta_{<\alpha} , \beta_{\alpha}]\cap S) =_T \omega^\omega$ for each $\alpha \in \omega_1$. Since $D$ is a closed subset of $S \Delta T$, for each limit ordinal $\alpha$, $\beta_{<\alpha}$ is either in both $S$ and $T$ or in neither. 

For each $K \in \K(S)$, there exists the smallest $\alpha_K \in \omega_1$ such that $K= \bigcup_{\gamma < \alpha_K} K\cap [\beta_{<\gamma}, \beta_{\gamma}]$. Clearly, each $K\cap [\beta_{<\gamma}, \beta_{\gamma}] \in \K([\beta_{<\gamma}, \beta_{\gamma}] \cap S)$. Fix Tukey quotient maps $\phi_\alpha : \K([\beta_{<\alpha}, \beta_{\alpha}] \cap S) \to \K([\beta_{<\alpha}, \beta_{\alpha}] \cap T)$ for each $\alpha$ (assume that for each $\alpha$, $\phi_\alpha (\emptyset) = \emptyset$). Define $\phi : \K(S) \to \K(T)$ by $\phi (K) = \overline{\bigcup_{\gamma < \alpha_K} \phi_\gamma(K\cap [\beta_{<\gamma}, \beta_{\gamma}])}$. Then since for each limit ordinal $\alpha$, $\beta_{<\alpha}$ is either in both $S$ and $T$ or in neither, we have $\phi (K) \in \K(T)$. Clearly, $\phi$ is order-preserving. 

To show that $\phi (\K(S))$ is cofinal in $\K(T)$, let $L \in \K(T)$. Then there is $\alpha_L$ such that $K= \bigcup_{\gamma < \alpha_L} L\cap [\beta_{<\gamma}, \beta_{\gamma}]$. Since each $\phi_\gamma$ is cofinal, there is $K_\gamma \in \K(S\cap [\beta_{<\gamma}, \beta_{\gamma}])$ such that $L\cap [\beta_{<\gamma}, \beta_{\gamma}] \subseteq \phi_\gamma (K_\gamma)$. If $L\cap [\beta_{<\gamma}, \beta_{\gamma}] = \emptyset$ then we may choose $K_\gamma = \emptyset$. Then since for each limit ordinal $\alpha$, $\beta_{<\alpha}$ is either in both $S$ and $T$ or in neither, we get that $K= \overline{\bigcup_{\gamma < \alpha_L} K_\gamma}$ is a compact subset of $S$ and $L\subseteq \phi(K)$. 
\end{proof}

\begin{cor}
Let $S$ and $T$ be both stationary and co-stationary and suppose $S \Delta T$ is non-stationary. Then $\K(S) \te \K(T)$.
\end{cor}



The Tukey classes in $\K(\SK)$ and their relations are summarized in Figure~\ref{KSclasses_omega_1<b}. The lines indicate that node to the right is strictly Tukey-above the one on the left. Solid lines indicate there is nothing strictly between the connected classes. Text in the boxes describes the corresponding equivalence classes. Note that the `stationary, co-stationary' category contains the maximal antichain in Todor\v{c}evi\'{c}'s theorem.


\begin{center}
 \resizebox{1\textwidth}{!}{

\begin{tikzpicture}
 \path node (a) at ( 0,6)  {$1$}
       node (b) at (2,8 )   {$\omega$}
       node (c) at (2,4)   {$\omega_1$}
	node (d) at (5,6) {$\omega_1\times \omega$}
	node (e) at (6,9) {$\omega^\omega$}
	node (f) at (8.5,7.5) {$\omega_1 \times \omega^\omega$}
        node (g) at (10,3.5) {$[\omega_1]^{<\omega}$}    
        node (h) at (15,5) {$[\omega_1]^{<\omega}\times \omega^\omega$}    
        node (q) at (11.5,6)  [text width=1.8cm] {\tiny{S stationary, co-stationary}} 
        node (r) at (14,7.5)  {$\Sigma (\omega^{\omega_1})$}

        node (i) at (0.5,4) [shape=rectangle, draw] {\tiny{S compact}}
        node (j) at (1,10) [shape=rectangle, draw] {\tiny{S bounded, $\cl{S}\backslash S\neq \emptyset$ closed}}
        node (k) at (2.5,2) [shape=rectangle, draw] {\tiny{S closed, unbounded}}
        node (l) at (3.2,11) [shape=rectangle, draw] {\tiny{S unbounded, $\cl{S}\backslash S\neq \emptyset$ closed and bounded}}
        node (m) at (8,2) [shape=rectangle, draw] {\tiny{S unbounded, non-stationary, $\cl{S}\backslash S\neq \emptyset$ closed}}
        node (n) at (13.1,3) [shape=rectangle, draw] {\tiny{S unbounded, non-stationary, $\cl{S}\backslash S$ not closed}}
        
        node (o) at (8,11) [shape=rectangle, draw] {\tiny{S bounded, $\cl{S}\backslash S$ not closed}}
        
      node (p) at (10.5,10.3) [shape=rectangle, draw] {\tiny{S unbounded, $\cl{S}\backslash S$ not closed but bounded}}
      node (s) at (13.1,9.5)  [shape=rectangle, draw] {\tiny{S stationary, not co-stationary, $\cl{S}\backslash S$ unbounded}} 
;

 \draw  (a) -- (c)                ;
 \draw  (d) -- (g)                ;
  \draw  (a) -- (b)                ;
   \draw  (b) -- (d)                ;
    \draw  (c) -- (d)                ;
     \draw  (b) -- (e)                ;
       \draw  (d) -- (f)                ;
        \draw  (e) -- (f)    node[text centered, midway, above, sloped] {\tiny{$=_T$ iff $\omega_1 =\mathfrak{b}$}}        ;
         \draw  (g) -- (h)    node[text centered, midway, above, sloped] {\tiny{$=_T$ iff $\omega_1 =\mathfrak{d}$}}    ;
          \draw [dashed] (f) -- (q)                ;
           \draw  (q) -- (h)                ;
            \draw  (f) -- (r)           ;
           \draw [dashed] (r) -- (h)                ;
             
              

   \draw [snake=snake] (a) -- (i)                ;   
   \draw [snake=snake] (b) -- (j)                ;  
   \draw [snake=snake] (c) -- (k)                ;        
    \draw [snake=snake] (d) -- (l)                ; 
     \draw [snake=snake] (g) -- (m)                ; 
      \draw [snake=snake] (h) -- (n)                ;   
      \draw [snake=snake] (e) -- (o)                ;   
      \draw [snake=snake] (f) -- (p)                ;  
       \draw [snake=snake] (r) -- (s)                ;       
             

\end{tikzpicture}}


\smallskip

Figure~\ref{KSclasses_omega_1<b}. Classes of $\K(S)$\label{KSclasses_omega_1<b}

\end{center}

It remains to justify the Tukey relations  claimed in Figure~\ref{KSclasses_omega_1<b}. Some of these Tukey relations  follow easily from results given so far. For example, by Theorem~\ref{M_initial}, $1 <_T \omega <_T \omega^\omega$. Since all countable subsets of $\omega_1$ are bounded, $\omega$ and $\omega_1$ are Tukey-incomparable and $\omega, \omega_1 <_T \omega_1 \times \omega$. It is clear that $\omega^\omega \times [\omega_1]^{<\omega} \tr  [\omega_1]^{<\omega}$ and,   by Lemma~\ref{cof}, $[\omega_1]^{<\omega}$ and $\omega^\omega \times [\omega_1]^{<\omega}$ are Tukey equivalent if and only if  $\omega_1=\mathfrak{d}$. Also, by Lemma~\ref{cof}, $\omega^\omega \leq_T [\omega_1]^{<\omega}$  if and only if  $\omega_1=\mathfrak{d}$. Clearly, $\omega^\omega \times \omega_1 \tr \omega^\omega$ and since $\omega^\omega \tr \omega_1$ if and only if $\omega_1 = \mathfrak{b}$, we have that $\omega^\omega \times \omega_1 \te \omega^\omega$ if and only if $\omega_1 = \mathfrak{b}$. By Lemma~\ref{cal1,0},  $[\omega_1]^{<\omega} \nleq_T \Sigma (\omega^{\omega_1})$ and for stationary, co-stationary $S$, $[\omega_1]^{<\omega} \nleq_T \K(S)$. Other relations require a little more work.

\begin{lemma} \ 

\begin{itemize}
\item[(a)] $\omega_1\times \omega <_T \omega_1\times \omega^\omega$.
\item[(b)] $\omega \times \omega_1 <_T  [\omega_1]^{<\omega}$.
\item[(c)] $\omega_1 \times \omega \ngeq_T \omega^\omega$ and $\omega_1 \times \omega <_T \omega^\omega$ if and only if $\omega_1 = \mathfrak{b}$.
\item[(d)] $\omega_1 \times \omega^\omega \ngeq_T [\omega_1]^{<\omega}$ and $\omega_1 \times \omega^\omega <_T [\omega_1]^{<\omega}$ if and only if $\omega_1 = \mathfrak{d}$.

\item[(e)] $ \omega_1\times \omega^\omega \leq_T \Sigma (\omega^{\omega_1}) <_T [\omega_1]^{<\omega} \times \omega^\omega$.
\item[(f)] If $S$ is stationary and co-stationary, then $\omega_1\times \omega^\omega <_T \K(S) <_T [\omega_1]^{<\omega} \times \omega^\omega$.
\item[(g)] If $S$ is stationary and co-stationary, then $\K(S) \nleq_T \Sigma (\omega^{\omega_1})$.
\end{itemize}
\end{lemma}
\begin{proof}
(a) The map $\psi:\omega_1 \times \omega \to \omega_1\times \omega^\omega$ defined by $\psi((\alpha, n)) = (\alpha, (n,0,0,0,\ldots))$ witnesses $\omega_1 \times \omega \leq_T \omega_1\times \omega^\omega$. Since both $\omega$ and $\omega_1$ have calibre $(\omega_1, \omega_1, \omega)$, $\omega_1\times \omega$ must also have this calibre. So, if $\omega_1 \times \omega \tr \omega_1\times \omega^\omega$, $\omega_1 \times \omega^\omega$ must have calibre $(\omega_1, \omega_1, \omega)$ as well. But $\omega_1 \times \omega^\omega$ has calibre $(\omega_1, \omega_1, \omega)$ if and only if $\omega^\omega$ does, which happens if and only if $\omega_1 < \mathfrak{b}$. However, when $\omega_1 < \mathfrak{b}$, $\cof{\omega_1 \times \omega} = \aleph_1 < \mathfrak{d} =  \cof{\omega_1 \times \omega^\omega}$ and we cannot have $\omega_1\times \omega \tr \omega_1 \times \omega^\omega$.

(b) Clearly, $\omega <_T  [\omega_1]^{<\omega}$. The map $\psi : \omega_1 \to [\omega_1]^{<\omega}$ defined by $\psi(\alpha) = \{\alpha\}$ is a Tukey map. So, we have $\omega \times \omega_1 \leq_T  [\omega_1]^{<\omega}$. For strict inequality, recall that $[\omega_1]^{<\omega}$ does not have calibre $(\omega_1, \omega)$, while $\omega \times \omega_1$ clearly does.

(c) Since $\omega_1\times \omega \tr \omega_1$ and $\omega_1$, $\omega^\omega$ are Dedekind complete, $\omega_1 \times  \omega \tr \omega^\omega$ implies $\omega_1 \times \omega \tr \omega_1 \times \omega^\omega$, which is not true. So $\omega_1 \times \omega \ngeq_T \omega^\omega$. Also, since $\omega <_T \omega^\omega$, $\omega_1 \times \omega <_T \omega^\omega$ if and only if $\omega_1 <_T \omega^\omega$, which happens if and only if $\omega_1 = \mathfrak{b}$.

(d) Since every countable subset of $\omega_1$ is bounded and $ \omega^\omega$ has calibre  $(\omega_1, \omega)$, $\omega_1 \times \omega^\omega$ also has calibre $(\omega_1, \omega)$, while $[\omega_1]^{<\omega}$ does not have it, hence $\omega_1 \times \omega^\omega \ngeq_T [\omega_1]^{<\omega}$. We know $\omega_1 \times \omega^\omega <_T [\omega_1]^{<\omega}$ if and only if $\cof{\omega_1 \times \omega^\omega} \leq \omega_1$, which happens if and only if $\omega_1 = \mathfrak{d}$.

(e)  We have already proved that $\Sigma (\omega^{\omega_1}) \te \K(S)$, where $S$ is stationary, not co-stationary and $\cl{S} \backslash S$ is unbounded. In this case $\cl{S} \backslash S$ is not closed and therefore $S$ contains a metric fan as a closed subset. So $\omega^\omega \leq_T \Sigma (\omega^{\omega_1})$. On the other hand, for every unbounded $S$, $\omega_1 \leq_T \K(S)$. So, $\omega_1\times \omega^\omega \leq_T \Sigma (\omega^{\omega_1})$. The inequality is strict because $ [\omega_1]^{<\omega} \leq_T \K(S)$ if and only if $S$ is not stationary.  

(f) For a stationary, co-stationary $S$, $\cl{S} \backslash S$, is not closed and by the same argument as in (j), $\omega_1\times \omega^\omega \leq_T \K(S) <_T [\omega_1]^{<\omega} \times \omega^\omega$. To show that the first inequality is also strict, recall that $\omega_1 \times \omega^\omega \leq_T \K(\Q)$ \cite{Kcal}. Therefore $\K(S) \leq_T \omega_1 \times \omega^\omega$ implies $(S, \K(S)) \leq_T \K(\Q)$, which contradicts Lemma \ref{KMoverS}. 

(g) Let $S'$ be stationary and not co-stationary and suppose $\K(S) \leq_T \K(S')$. Then by Todor\v{c}evi\'{c}'s theorem $S'\backslash S$ must be non-stationary. But $S'\backslash S = S' \cap \omega_1 \backslash S$ and since $S'$ contains a cub set and $\omega_1 \backslash S$ is stationary, their intersection should also be stationary.  Therefore we get $\K(S) \nleq_T \K(S') \te \Sigma (\omega^{\omega_1})$. 
\end{proof}

Lastly we have to prove that $\omega^\omega <_T \Sigma (\omega^{\omega_1})$ and $\omega_1\times \omega^\omega <_T \Sigma (\omega^{\omega_1})$. Recalling that $\omega^\omega \te \K(\omega^\omega )$, we show something more general than $\omega^\omega <_T \Sigma (\omega^{\omega_1})$ in Proposition~\ref{Sigma_M}. Proposition~\ref{Sigma_M} distinguishes between two of Isbell's ten posets. Since $\K(\Q) \tr \omega_1\times \omega^\omega$, Proposition~\ref{Sigma_M} immediately implies that  $\omega_1\times \omega^\omega <_T \Sigma (\omega^{\omega_1})$. 

\begin{prop}\label{Sigma_M}
For any separable metrizable $M$, $\K(M) \ngeq_T \Sigma (\omega^{\omega_1})$. 
\end{prop}

\begin{proof}
For a subset $A$, and element $p$,  of a poset $P$, we write $A \le p$ if $p$ is an upper bound of $A$.
Suppose for a contradiction that $M$ is separable metrizable, and $\phi : \K(M) \to \SigOm$ is order-preserving and cofinal. Let $\B$ be a countable base for $M$ closed under finite unions. 

\begin{claim}\label{Key}
Given $\alpha$ in $\omega_1$ and $K$ in $\K(M)$, there is a $B$ in $\B$ containing $K$ and $m$ such that $\phi(\K(B))(\alpha) \le m$.
\end{claim}

\begin{proof} 
Fix the $\alpha$ and $K$. Fix a descending sequence $B_1 \supseteq B_2 \supseteq \cdots$ in $\B$ forming a base for $K$ in $M$. If the claim were false then for each $n$ there would be a $K_n$, a compact subset of $B_n$, such that $\phi(K_n)(\alpha) >n$. Set $K_\infty=K \cup \bigcup_n K_n$. Then $K_\infty$ is compact, and is an upper bound of all the $K_n$'s. As $\phi$ is order-preserving, $\phi(K_\infty)$ is an upper bound  of the $\phi(K_n)$'s. But $\phi(K_n)(\alpha) >n$, for each $n$. Contradiction.
\end{proof}

Define, for $B$ in $\B$, $m$ in $\omega$ and $\alpha$ in $\omega_1$:
\[ T_{B,m} = \{ \alpha \in \omega_1 : \phi(\K(B))(\alpha) \le m\}, \text{ and } F_\alpha=\bigcap \{ T_{B,m} : \alpha \in T_{B,m}\}.\]

\begin{claim}\label{Fa_fin}
For every $\alpha$ the set $F_\alpha$ is finite and contains $\alpha$. 
\end{claim}
\begin{proof} Fix $\alpha$. 
Clearly $\alpha$ is in $F_\alpha$, for every $\alpha$. 
If $F_\alpha$ were infinite then there would be a subset $\{\delta_n : n \in \omega\}$ of $F_\alpha$ where $\delta_n \ne \delta_{n'}$ if $n \ne n'$. Define $\tau$ in $\SigOm$ by $\tau(\delta_n)=n$ for each $n$, and zero elsewhere. Since $\phi$ is cofinal there is a $K$ in $\K(M)$ such that $\phi(K) \ge \tau$. By Claim \ref{Key}, there is a $B$ in $\B$ containing $K$ and $m$ such that $\phi(\K(B))(\alpha) \le m$. For each $n$, since $\delta_n$ is in $F_\alpha$ we must have $\phi(\K(B))(\delta_n) \le m$. But $K$ is in $\K(B)$ and $\phi(K)(\delta_n) \ge \tau (\delta_n)=n$ for every $n$ in $\omega$. Contradiction.
\end{proof}

\begin{claim}\label{Cint_fin}
For any $\alpha$ and countably infinite $C \supseteq F_\alpha$ there is a $B$ in $\B$ and $m$ such that $F_\alpha \subseteq T_{B,m}$ and $T_{B,m} \cap C$ is finite. 
\end{claim}

\begin{proof} Fix the $\alpha$ and $C$. Write $C=\{\gamma_n : n \in \omega\}$. Define $\tau$ in $\SigOm$ by $\tau(\gamma_n)=n$, and is otherwise zero. As $\phi$ is cofinal, there is a $K$ in $\K(M)$ such that $\phi(K) \ge \tau$. By Claim~\ref{Key} there is a $B$ in $\B$ containing $K$ and $m$ such that $\phi(\K(B))(\alpha) \le m$. By definition of $T_{B,m}$ we see that $\alpha$ is in $T_{B,m}$, and hence, by definition of $F_\alpha$, clearly $F_\alpha \subseteq T_{B,m}$. For $n>m$ we have $\phi(K)(\gamma_n) \ge n >m$. Since $K \in \K(B)$ it follows that for $n>m$ no $\gamma_n$  is in $T_{B,m}$, and hence $T_{B,m} \cap C$ is finite.
\end{proof}

To complete the proof of Proposition~\ref{Sigma_M} we show that the claimed properties of the $T_{B,m}$'s and $F_\alpha$'s lead to the desired contradiction. Since $\alpha \in F_\alpha$ and the $F_\alpha$'s are finite (Claim~\ref{Fa_fin}), we can recursively pick $\alpha_\delta$ for each $\delta$ in $\omega_1$, an increasing sequence such that $\alpha_\delta \in F_{\alpha_\delta}$ but $\alpha_\delta \notin \bigcup_{\gamma < \delta} F_{\alpha_\gamma}$. Let $C_\delta = \bigcup_{\gamma \le \delta} F_{\alpha_\gamma}$. Note that $C_\delta$ is countably infinite and contains $\alpha_\delta$.

For each $\delta$ we know  (Claim~\ref{Cint_fin}) there is a $B_\delta$ and $m_\delta$ such that $\alpha_\delta \in T_{B_\delta,m_\delta}$ and $T_{B_\delta, m_\delta} \cap C_\delta$ is finite. Since $\B$ is countable there is an uncountable $S \subseteq \omega_1$, a $B$ in $\B$ and $m$ such that $B_\delta=B$ and $m_\delta=m$ for all $\delta$ in $S$. Let $\{\delta_n : n \in \omega\}$ be an infinite subset of $S$ and $\delta_\infty$ in $S$ such that $\delta_\infty \ge \delta_n$ for all $n$. Then $T_{B,m} \cap C_{\delta_\infty} \supseteq T_{B,m} \cap C_{\delta_n}$ contains $\alpha_{\delta_n}$ for every $n$, but $T_{B,m} \cap C_{\delta_\infty}$ is finite. Contradiction. \end{proof}

\section{Applications}\label{comparing_M_S}

\subsection{Elements of $\K(\SK)$ Tukey-below some $\K(M)$} Work on the Tukey ordering pursued by Cascades, Orihuela and Tkachuk focuses on understanding spaces $X$ such that $\K(X)$ sits Tukey-below $\K(M)$ for some separable metrizable space $M$, or more generally, when $X$ has a compact cover that is ordered by $\K(M)$ (i.e. $\K(M) \tr (X, \K(X))$). In this subsection we will show precisely which subsets of $\omega_1$ satisfy these conditions. 

It turns out that the condition $\K(M) \tr \K(S)$ is rather restrictive. Since $\K(M)$ always has calibre $(\omega_1, \omega)$, it is not possible to have $\K(M) \tr [\omega_1]^{<\omega} \times \omega^\omega$ or $\K(M)\tr [\omega_1]^{<\omega}$. Proposition~\ref{Sigma_M} says that $\K(M) \tr \Sigma (\omega^{\omega_1})$  never happens. The next lemma further narrows down the possibilities. 

\begin{lemma}\label{KMoverS}
Suppose $S\subseteq \omega_1$ is unbounded and there is separable metric $M$ such that $\K(M)\geq_T (S,\K(S))$. Then $S$ is not co-stationary. 
\end{lemma}
\begin{proof}
Suppose $\phi : \K(M) \rightarrow \K(S) $ is order-preserving and the image of $\phi$ covers $S$. Then as in Proposition~2.6 of \cite{cot} let $\B$ be a countable base of $M$ that is closed under finite unions and finite intersections and for each $B\in \B$ define $G(B)=\bigcup \{ \phi(K) : K\subseteq B, \ B\in \K(M) \}$. 

There is $x \in M$ such that for each $x\in B\in \B$, $G(B)$ is unbounded in $\omega_1$. Otherwise $\B' = \{ B\in \B : \ G(B) \ \text{is bounded} \}$ is also a base of $X$ that is closed under finite intersections and unions. Therefore the $G(B)$'s cover $\omega_1$, but this is a contradiction since there are only countably many of them. 

For a cardinal $\theta$ let $H(\theta)$ be the set of all sets with $<\theta$-sized transitive closure. We know that if $\theta$ is regular and uncountable, all axioms of ZFC, with the exception of the Power Set Axiom, are true in $H(\theta)$. Suppose $S$ is co-stationary and $\theta$ is a regular cardinal large enough so that $H(\theta)$ contains all sets we need in this argument. As in the proof of Lemma~1 in \cite{Tod1}, let $E$ be a countable elementary submodel of $H(\theta)$ such that $\phi, S, M, \B, G \in E$ and $\omega_1 \cap E \in \omega_1 \backslash S$. 


By elementarity there is $x\in M\cap E$ with decreasing local base $\{B_n : n\in \omega \} \subseteq \B$ at $x$ such that each $G(B_n)$ is unbounded in $\omega_1$. Then, by elementarity and since $G(B_n)$ is unbounded, for each $n$ and $\alpha \in \omega_1 \cap E$ there is $K_{n,\alpha} \in \K(M)\cap E$ such that $K_{n,\alpha} \subseteq B_n$ and $\sup(\phi(K_{n,\alpha})) \geq \alpha$. Then $K_{n,\alpha} \in E$ and $K_{n,\alpha}$ is countable. So, $K_{n,\alpha} \subseteq \omega_1 \cap E$. Pick $\{ \alpha_n : n\in \omega_1\} $ such that $\{ \alpha_n \}_{n}$ converges to $ \omega_1 \cap E$ and let $K=\{x\} \cup \bigcup_{n\in \omega}K_{n,\alpha_n}$. Then $K\in \K(M)$ and $\phi(K_{n,\alpha_n}) \subseteq \phi(K)$ for each $n$. But this contradicts $ \phi(K) \in \K(S)$ and $\omega_1 \cap N \notin S$. 
\end{proof}

\begin{prop}
Let $S$ be a subset of $\omega_1$ that contains a cub set. Then $\omega_1 \times \omega \tr (S, \K(S))$. Hence $\K(\Q) \tr (S, \K(S))$.
\end{prop}
\begin{proof}
Fix $S\subseteq \omega_1$ and a cub set $C \subseteq S$. Let $C=\{ \beta_\alpha : \alpha \in \omega_1 \}$ be the increasing enumeration of $C$. For each $\alpha \in \omega_1$ enumerate $[\beta_\alpha, \beta_{\alpha+1}] \cap S$ as  $\{ x_{\alpha, n}: n \in \omega \}$, with repetitions if necessary, and let $F_{\alpha, n}= \{ x_{\alpha,0}, x_{\alpha, 1}, \ldots , x_{\alpha, n} \}$.

Define $\phi : \omega_1 \times \omega \to \K(S)$ by $\phi((\alpha, n)) = \cl{\bigcup_{\gamma \leq \alpha} F_{\gamma, n}}$. Since $C$ is a cub set, the only limit points of $\bigcup_{\gamma \leq \alpha} F_{\gamma, n}$ outside $\bigcup_{\gamma \leq \alpha} F_{\gamma, n}$ are in $C$, so $\phi((\alpha, n))$ is indeed in $\K(S)$. Clearly, $\phi$ is order-preserving and the image covers $S$. 
\end{proof}

\begin{cor}
For unbounded $S \subseteq \omega_1$, there exists a separable metrizable $M$ with $\K(M)\tr S$ if and only if $S$ is in the cub filter. 
\end{cor}

\begin{cor}
For $S\subseteq \omega_1$, there exists a separable metrizable $M$ with $\K(M) \tr \K(S)$ if and only if $\cl{S} \backslash S$ is bounded. 
\end{cor}
\begin{proof}
If $\cl{S} \backslash S$ is bounded then $S = C\oplus N$, where $C$ is a closed unbounded set or an empty set and $N$ is countable (i.e. $\cl{S} \backslash S$ is bounded). Then, since cub sets are homeomorphic to $\omega_1$, $\K(\Q) \tr \K(C)$. Since $N$ is Polish, $\omega^\omega \tr \K(N)$. Now set $M = \Q$ and we have $\K(S) = \K(C) \times \K(N) \leq_T \K(\Q) \times \K(\omega^\omega) \leq_T \K(\Q) \times \K(\Q) \te \K(\Q)$.

On the other hand, if $\K(M) \tr \K(S)$ for some $M$, then $S$ contains a closed unbounded set. If, in addition, $\cl{S} \backslash S$ is unbounded then $\K(S) \te \Sigma(\omega^{\omega_1})$, which contradicts  Proposition~\ref{Sigma_M}.
\end{proof}

\subsection{Function Spaces}

For any space $X$ let $C(X)$ be the set of all real-valued continuous functions on $X$. Let $\zerof$ be the constant zero function on $X$. For any function $f$ from $C(X)$, subset $S$ of $X$ and $\epsilon>0$ let $B(f,S,\epsilon)=\{ g \in C(X) : |f(x)-g(x)| < \epsilon \ \forall x \in S\}$. Write $C_p(X)$ for $C(X)$ with the pointwise topology (so basic neighborhoods of an $f$ in $C_p(X)$ have the form $B(f,F,\epsilon)$ where $F$ is finite and $\epsilon >0$). Write $C_k(X)$ for $C(X)$ with the compact-open topology (so basic neighborhoods of an $f$ in $C_k(X)$ have the form $B(f,K,\epsilon)$ where $K$ is compact and $\epsilon >0$). In \cite{KM} the authors established connections between the function spaces $C_p(X)$ and $C_k(X)$ and the poset $\K(X)$. Using these connections the authors also obtained $2^\ctm$-sized families of functions spaces from a $2^\ctm$-sized family of pairwise Tukey-incomparable $\K(M)$'s where $M$'s are separable metrizable spaces. Here we derive similar results for subsets of $\omega_1$. 

A space $X$ is strongly $\omega$-bounded  if and only if  whenever $\{ K_n : n \in \omega\}$ is a countable family of compact subsets of $X$, then $\cl{\bigcup \{ K_n : n \in \omega\}}$ is compact. Let $\F(X) = \{ F\subseteq X : F \ \text{is finite} \} \subseteq \K(X)$. The following result was proven in \cite{KM}.

\begin{prop}\label{cp_ck}
Let $X$ and $Y$ be spaces.

(1) Suppose that $X$ is not strongly $\omega$-bounded and there is a linear embedding of $C_p(Y)$ into $C_p(X)$. Then  $(\F(X),\K(X)) \tr (\F(Y),\K(Y))$.

(2) Suppose that either there is a continuous open surjection of $C_k(X)$ onto $C_k(Y)$ or $C_k(Y)$ embeds in $C_k(X)$. Then  $\K(X) \times \omega\tr \K(Y) \times \omega$, and if neither $X$ nor $Y$ are strongly $\omega$-bounded spaces then $\K(X) \tr \K(Y)$.
\end{prop}

\begin{lemma}
Let $S$ and $T$ be subsets of $\omega_1$. If $S$ and $T$ are co-stationary and there is a continuous linear surjection of $C_p(S)$ onto $C_p(T)$ then (a) $\K(S) \tr \K(T)$ and (b) $(\F(S), \K(S)) \tr (\F(T), \K(T))$. 
\end{lemma}
\begin{proof}
Let $X$ and $Y$ be any spaces such that there is a continuous linear surjection from $C_p(S)$ onto $C_p(T)$. Then for any compact $L\subseteq Y$, $\cl{\supp{L}}$ is a compact subset of $X$ and for any closed and functionally bounded $K\subseteq X$, $L=\{ y\in Y : \supp{y} \subseteq K \}$ is closed and functionally bounded \cite{Arh}. Here $A \subseteq X$ is called functionally bounded if and only if $f(A)$ is bounded for any $f\in C_p(X)$. For subsets of co-stationary $S\subseteq \omega_1$ being closed and functionally bounded  is equivalent to being compact. To see this, take a closed subset of $S$, say $C$. If $C$ is not closed in $\omega_1$, then $C$ contains an increasing sequence that converges to a point outside $S$ and we can find $f\in C_p(S)$ such that $f(C)$ is unbounded. Therefore, $C$ must be closed in $\omega_1$ and since $S$ is co-stationary it must be bounded. So $C$ is compact. Now the map $\phi : \K(S) \to \K(T)$ defined by $K \mapsto \{ y\in Y : \supp{y} \subseteq K \}$ is well-defined, order-preserving and since for each $L\in \K(T)$, $L\subseteq \phi(\cl{\supp{L}})$, it is also cofinal. From the definition it is clear that $\phi(\F(S))$ is cofinal for $\F(T)$ in $\K(T)$, which establishes part (b).
\end{proof}

Now we apply Theorem~\ref{antichain}. It is easy to see (and was shown in \cite{KM}) that $\K(X) \ngeq_T (Y, \K(Y))$ implies $(\F(X),\K(X)) \ngeq_T (\F(Y),\K(Y))$.  Notice that whenever $S\subseteq \omega_1$ is not closed, $S$ is not strongly $\omega$-bounded. Since the $2^{\omega_1}$-sized family of subsets of $\omega_1$ from Theorem~\ref{antichain} consists of stationary and co-stationary (therefore not closed) subsets of $\omega_1$, we have the following corollary.

\begin{cor}
There is a $2^{\omega_1}$-sized family $\mathcal{A}$ of subsets of $\omega_1$ such that:

\begin{itemize}
\item[(1)] if $S$ and $T$ are distinct elements of $\mathcal{A}$, then $C_k(S)$ is not a continuous open image of $C_k(T)$ and does not embed in $C_k(T)$, and
\item[(2)] if $S$ and $T$ are distinct elements of $\mathcal{A}$, then there is no linear surjection of $C_p(S)$ onto $C_p(T)$ and no linear embedding of $C_p(S)$ in $C_p(T)$.
\end{itemize}
\end{cor}

\end{document}